\documentclass{eptcs-modified}

\usepackage{amsmath}
\usepackage{amssymb}
\usepackage{url}
\usepackage{amsthm}
\usepackage{graphicx}
\usepackage{amscd}
\usepackage[usenames,dvipsnames,svgnames,table]{xcolor}
\usepackage{tikz}
\usetikzlibrary{arrows,chains,matrix,positioning,scopes}
\makeatletter
\tikzset{join/.code=\tikzset{after node path={%
\ifx\tikzchainprevious\pgfutil@empty\else(\tikzchainprevious)%
edge[every join]#1(\tikzchaincurrent)\fi}}}
\makeatother
\tikzset{>=stealth',every on chain/.append style={join},
         every join/.style={->}}
\tikzstyle{labeled}=[execute at begin node=$\scriptstyle,
   execute at end node=$]

\begin{document}

\theoremstyle{definition}
\newtheorem{theorem}{Theorem}
\newtheorem{definition}[theorem]{Definition}
\newtheorem{problem}[theorem]{Problem}
\newtheorem{assumption}[theorem]{Assumption}
\newtheorem{corollary}[theorem]{Corollary}
\newtheorem{proposition}[theorem]{Proposition}
\newtheorem{example}[theorem]{Example}
\newtheorem{lemma}[theorem]{Lemma}
\newtheorem{observation}[theorem]{Observation}
\newtheorem{fact}[theorem]{Fact}
\newtheorem{question}[theorem]{Open Question}
\newtheorem{conjecture}[theorem]{Conjecture}
\newtheorem{addendum}[theorem]{Addendum}
\newtheorem{remark}[theorem]{Remark}
\newcommand{\uint}{{[0, 1]}}
\newcommand{\Cantor}{{\{0,1\}^\mathbb{N}}}
\newcommand{\name}[1]{\textsc{#1}}
\newcommand{\me}{\name{P.}}
\newcommand{\id}{\textrm{id}}
\newcommand{\dom}{\operatorname{dom}}
\newcommand{\Dom}{\operatorname{Dom}}
\newcommand{\codom}{\operatorname{CDom}}
\newcommand{\spec}{\operatorname{spec}}
\newcommand{\opti}{\operatorname{Opti}}
\newcommand{\optis}{\operatorname{Opti}_s}
\newcommand{\Baire}{\mathbb{N}^\mathbb{N}}
\newcommand{\hide}[1]{}
\newcommand{\mto}{\rightrightarrows}
\newcommand{\Sierp}{Sierpi\'nski }
\newcommand{\BC}{\mathcal{B}}
\newcommand{\C}{\textrm{C}}
\newcommand{\lpo}{\textrm{LPO}}
\newcommand{\llpo}{\textrm{LLPO}}
\newcommand{\leqW}{\leq_{\textrm{W}}}
\newcommand{\leW}{<_{\textrm{W}}}
\newcommand{\equivW}{\equiv_{\textrm{W}}}
\newcommand{\equivT}{\equiv_{\textrm{T}}}
\newcommand{\geqW}{\geq_{\textrm{W}}}
\newcommand{\pipeW}{|_{\textrm{W}}}
\newcommand{\nleqW}{\nleq_\textrm{W}}
\newcommand{\leqsW}{\leq_{\textrm{sW}}}
\newcommand{\equivsW}{\equiv_{\textrm{sW}}}
\newcommand{\Sort}{\operatorname{Sort}}
\newcommand{\RT}{\mathrm{RT}}
\newcommand{\TT}{\mathrm{TT}^1}
\newcommand{\tcn}{\mathrm{TC}_\mathbb{N}}
\newcommand{\accn}{\mathrm{ACC}_\mathbb{N}}
\newcommand{\ecfc}{\mathrm{eCFC}_\mathbb{N}}
\newcommand{\Fin}{\operatorname{Fin}}

\newcommand{\pitc}{\Pi^0_2\textrm{C}}

\title{More on the degree of indivisibility of $\mathbb{Q}$}

\author{
Arno Pauly
\institute{Swansea University\\Swansea, UK}
\email{Arno.M.Pauly@gmail.com}
}

\def\titlerunning{More on the degree of indivisibility of $\mathbb{Q}$}
\def\authorrunning{A. Pauly}
\maketitle

\begin{abstract}
We study the complexity of the computational task ``Given a colouring $c : \mathbb{Q} \to \mathbf{k}$, find a monochromatic $S \subseteq \mathbb{Q}$ such that $(S,<) \cong (\mathbb{Q},<)$''. The framework is Weihrauch reducibility. Our results answer some open questions recently raised by Gill, and by Dzhafarov, Solomon and Valenti.
\end{abstract}

\section{Introduction}
We call a structure $\mathcal{M}$ over $\mathbb{N}$ \emph{indivisible}, if for every colouring of $\mathbb{N}$ with finitely many colours there is a monochromatic isomorphic copy of $\mathcal{M}$. A typical example of an indivisible structure is $(\mathbb{Q},<)$. For a fixed indivisible structure $\mathcal{M}$ we can then study the computational task $\mathrm{Ind}\mathcal{M}$, which receives as input a $k$-colouring of $\mathbb{N}$ and which has to output a monochromatic copy of $\mathcal{M}$. This programme was recently formulated by Kenneth Gill \cite{gill-phd,gill-arxiv}, who obtained results about the Weihrauch degree of $\mathrm{Ind}\mathbb{Q}$ and some other structures. In a largely independent development, Dzhafarov, Solomon and Valenti \cite{damir-reed-manlio} studied the Weihrauch degree of the tree pigeon hole principle $\mathrm{TT}^1_+$, which can be seen as indivisibility of the full binary tree with relations for ``is in the left subtree below'' and ``is in the right subtree below''. It is easy to see that $\mathrm{Ind}\mathbb{Q} \equivW \mathrm{TT}^1_+$ (and for any fixed number of colours, the corresponding restrictions are equivalent, too).

\section{Background}
For more context on divisibility and computability we refer to \cite{gill-phd,gill-arxiv}, for context on the tree pigeon principle to \cite{damir-reed-manlio}. The required background on Weihrauch reducibility is provided in \cite{pauly-handbook}.

Of particular relevance for our investigation are the pigeon hole principles $\RT^1_k$, which take as input some $c : \mathbb{N} \to \mathbf{k}$ and return some $i \in \mathbf{k}$ such that $c^{-1}(i)$ is infinite. We can view these principles as $\mathrm{Ind}(\mathbb{N},<)_k$. The notation stems from viewing them as the 1-dimensional case of Ramsey's theorem. We write $\RT^1_+$ for the principle where $k$ is not fixed, but provided as part of the input.

We also refer to the principle $\tcn$ from \cite{paulyneumann}, which takes as input an enumeration of the complement of some $A \subseteq \mathbb{N}$ as input, and returns some $n \in A$ if $A$ is non-empty, and some $n \in \mathbb{N}$ otherwise.

The first-order part of a Weihrauch degree $f$. denoted by $^1(f)$, is the maximal Weihrauch degree reducible to $f$ having a representative with codomain $\mathbb{N}$. The notion was proposed in \cite{damir-reed-keita}, and studied further in \cite{soldavalenti,pauly-valenti,paulysolda}. The investigation of $^1(\TT_k)$ was a key goal in \cite{damir-reed-manlio}. In particular, they prove that $^1(\TT_+) \leqW \RT^1_+ \star \C_\mathbb{N}$.

\section{Locating $\mathrm{TT}^1_k$}
We begin our investigation of $\TT_k$ by locating it in the Weihrauch lattice relative to some benchmark principles. The results are depicted in Figure \ref{figure:reducibilities}, which exhibits a grid-structure with several families of Weihrauch degrees parameterized by a natural number. There are Weihrauch reductions if we increase the parameter, or if we move to a more complicated family; but never from a principle with higher parameter to one with lower parameter, even if this is accompanied by changing the family. Figure \ref{figure:reducibilities} expands upon a similar figure from Gill's dissertation \cite{gill-phd} by adding the top row referring to $(\RT^1_\ell)'$ (also known as $\mathrm{D}^2_\ell$), and by completing the proof of absence of any additional reductions.

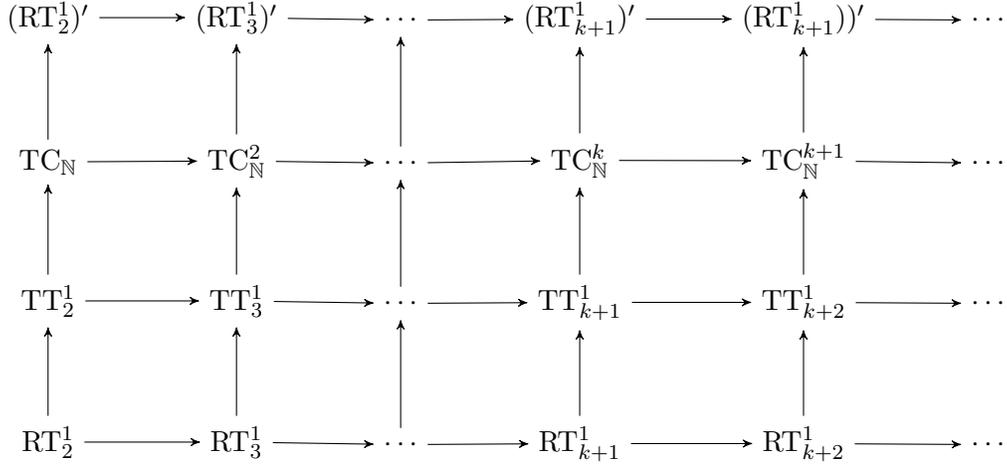
\begin{figure}[htbp]
\centering
\begin{tikzpicture}
  \matrix (m) [matrix of math nodes, row sep=3em, column sep=3em]
    { (\RT^1_2)'  & (\RT^1_3)'  & \cdots  & (\RT^1_{k+1})' & (\RT^1_{k+1}))' & \cdots \\
     \tcn  & \tcn^2  & \cdots  &  \tcn^k & \tcn^{k+1} & \cdots \\
      \TT_2  & \TT_3  & \cdots  & \TT_{k+1} & \TT_{k+2} & \cdots\\
      \RT^1_2 & \RT^1_3 & \cdots & \RT^1_{k+1} & \RT^1_{k+2} & \cdots\\ };
  { [start chain] \chainin (m-1-1);
      \chainin (m-1-2);
    \chainin (m-1-3);
    \chainin (m-1-4);
    \chainin (m-1-5);
        \chainin (m-1-6);}
  { [start chain] \chainin (m-2-1);
      \chainin (m-2-2);
    \chainin (m-2-3);
    \chainin (m-2-4);
    \chainin (m-2-5);
        \chainin (m-2-6);}
  { [start chain] \chainin (m-3-1);
      \chainin (m-3-2);
    \chainin (m-3-3);
    \chainin (m-3-4);
    \chainin (m-3-5);
        \chainin (m-3-6);}
      { [start chain] \chainin (m-4-1);
      { [start branch=E] \chainin (m-3-1); \chainin (m-2-1); \chainin (m-1-1);}
    \chainin (m-4-2);
      { [start branch=A] \chainin (m-3-2); \chainin (m-2-2); \chainin (m-1-2);}
    \chainin (m-4-3);
      { [start branch=B] \chainin (m-3-3); \chainin (m-2-3); \chainin (m-1-3);}
    \chainin (m-4-4);
           { [start branch=C] \chainin (m-3-4); \chainin (m-2-4); \chainin (m-1-4);}
    \chainin (m-4-5);
          { [start branch=D] \chainin (m-3-5); \chainin (m-2-5); \chainin (m-1-5);}
    \chainin (m-4-6); }
\end{tikzpicture}
\caption{All Weihrauch reductions (up to transitivity) between the principles $\RT^1_k$, $\TT_n$, $\tcn^j$ and $(\RT^1_\ell)'$.}
\label{figure:reducibilities}
\end{figure}

That $\RT^1_k \leqW \TT_k$ is immediate, as observed in both \cite{gill-arxiv} and \cite{damir-reed-manlio}. That $\TT_{k+1} \leqW \tcn^k$ is a corollary of results from \cite{paulypradicsolda}, but we include a direct proof here for the sake of self-containedness. As observed by Dzhafarov, Solomon and Valenti \cite{damir-reed-manlio}, one way to prove $\mathrm{TT}^1_2$ is to ask for a vertex below which the colouring is constant, if there is one. If we have such a vertex, we can built a monochromatic subtree by just greedily searching for vertices of the right colour. This yields the reduction $\mathrm{TT}^1_2 \leqW \mathrm{TC}_\mathbb{N}$.

\begin{proposition}
\label{prop:tttcn}
$\TT_{k+1} \leqW \tcn^k$
\begin{proof}
For the first $\tcn$, we consider an $\sigma_0 \in 2^{<\omega}$ and $I_0 \subseteq \{0,1,\ldots,k-1\}$ with $|I_0| = k- 1$ and enumerate all such pairs except for the currently first one where all vertices below $\sigma_0$ have colours in $I_0$. For the second instance, we consider $\sigma_1 \in 2^{<\omega}$ and $I_1 \subseteq \{0,1,\ldots,k-1\}$ with $|I_1| = k - 2$m and enumerate all such pairs except the currently first one where $\sigma_1$ is an extension of $\sigma_0$, $I_1 \subset I_0$, and all vertices below $\sigma_1$ have a colours in $I_1$. The pattern is repeated to produce the remaining inputs to $\tcn$.

From $\tcn^k$, we then obtain some $((\sigma_0,I_0),(\sigma_1,I_1),\ldots,(\sigma_{k-1},I_{k-1}))$. Let $j < k$ be maximal such that for all $i < j$ we find $\sigma_{i+1}$ to be an extension of $\sigma_i$ and $I_{i+1} \subset I_i$. We claim that picking any colour $b \in I_j$ and trying the greedy construction to built a $b$-monochromatic subtree below $\sigma_j$ will succeed.

To see this, first consider the information we obtain from the first $\tcn$-instance. If some colour is not dense everywhere, we will eventually encounter a combination $\sigma_0$ and $I_0$ which remains stable (because the colour in $\mathbf{k} \setminus I_0$ does not occur below $\sigma_0$). If all colours are dense everywhere, our greedy construction will succeed for sure. If the first instance for $\tcn$ settles on a specific input, then either all remaining $k-1$-many colours are dense below $\sigma_0$ (and by searching below $\sigma_0$ we already succeed), or the second input will eventually encounter a $\sigma_1$ below $\sigma_0$ and some $I_1 \subsetneq I_0$ where the removed colour does not appear below $\sigma_1$. We choose $\sigma_j$ and $I_j$ such that they definitely reflect all of the defined information, which is why we know that all colours in $I_j$ will be dense below $\sigma_j$.
\end{proof}
\end{proposition}

The jump of $\RT^1_k$ is often denoted by $\mathrm{D}^2_k$ in the literature on Weihrauch degrees or reverse mathematics of Ramsey-like theorems. We can think of $(\RT^1_k)'$ as receiving a $\Delta^0_2$-colouring $c : \mathbb{N} \to \mathbf{k}$ and having to return an infinite monochromatic subset. Knowing a suitable colour is insufficient to compute an infinite set of that colour, as we do not actually know the colour of a given number. In  \cite{damir-reed-manlio}, it is shown that $\TT_k \leqW (\RT^1_k)'$, which we shall improve to $\tcn^k \leqW (\RT^1_{k+1})'$ as follows:

\begin{proposition}
$\tcn^k \leqW (\RT^1_{k+1})'$
\begin{proof}
We are given $k$ $\tcn$-instances $A_1,\ldots,A_k$, and we may assume without loss of generality that each $A_i$ is either empty or cofinite. We define $c : \mathbb{N} \to (\mathbf{k+1})$ by letting $c(n) = |\{i \leq k \mid n \in A_i\}|$. We have co-enumerations of the sets $A_i$ available, which suffices to compute $c$ as a $\Delta^0_2$-colouring. If there are $\ell$ empty sets amongst the $A_i$, then the colours greater than $k - \ell$ never occur, and all but finitely many $n$ will receive the colour $k - \ell$. Moreover, any $n$ with $c(n) = k - \ell$ is included in all non-empty $A_i$, and thus is a correct to all the $\tcn$-instances at the same time. Thus, we can answer all $\tcn$-instances by returning the first element of the infinite $c$-homogeneous set we obtain from $(\RT^1_{k+1})'$.
\end{proof}
\end{proposition}

That $\RT^1_{k+1} \leW \TT_{k+1} \leW \tcn^k \leW (\RT^1_{k+1})'$ follows from known results, in particular ones from \cite{gill-arxiv,damir-reed-manlio}. Other than $(\RT^1_{k+1})' \nleqW \tcn^k$ (which holds because $\tcn^k$ can only give computable answers, while $(\RT^1_{k+1})'$ has a computable instance with no computable solution), the separations also follow from results in this article. That $\TT_2 \nleqW \RT^1_k$ is Corollary \ref{corr:rt1knleqttk}, that $\tcn \nleqW \TT_{k}$ is Corollary \ref{corr:cnnleqwttk}; both below.

To prove that in Figure \ref{figure:reducibilities} there are no additional reductions between the depicted principles, it is sufficient and necessary to show that $\RT^1_{k+1} \nleqW (\RT^1_k)'$ for all $k \in \mathbb{N}$. This follows from the jump inversion theorem for Weihrauch reducibility:

\begin{proposition}
$\RT^1_{k+1} \nleqW^* (\RT^1_k)'$ for all $k \in \mathbb{N}$.
\begin{proof}
The jump inversion theorem for Weihrauch reducibility (by Brattka, H\"olzl and Kuyper \cite{hoelzl2}) states that if $f' \leqW g'$ relative to oracle $p$, then $f \leqW g$ relative to $p'$. Since $\mathrm{RT}^1_{k+1} = (\C_{k+1})'$, we find that $\RT^1_{k+1} \leqW^* (\RT^1_k)'$ would imply $\C_{k+1} \leqW^* \RT^1_k$, which is false due to reasons of cardinality alone.
\end{proof}
\end{proposition}

\section{Finite guessability}

We will formalize the idea that for some problems we may not be able to compute a solution, however, we are able to compute finitely many guesses with the guarantee that at least one of them is correct. This comes in three variations: In one the number of guesses is fixed in advanced, in one the number of guesses is fixed at some point of the computation, and in the final the computation is always allowed to make one more guess, it only needs to stop doing so eventually (but never confirm that this is happening).

\begin{definition}
\begin{enumerate}
\item We call $f : \mathbf{X} \mto \mathbf{Y}$ $k$-guessable, if the map $\breve{f}_k : \mathbf{X} \mto \overline{\mathbf{Y}}^k$ with $(y_0,\ldots,y_{k-1}) \in \breve{f}_k(x)$ iff $\exists i < k \ y_i \in f(x)$ is computable.
\item We call $f : \mathbf{X} \mto \mathbf{Y}$ finitely guessable, if the map $\breve{f}_* : \mathbf{X} \mto \overline{\mathbf{Y}}^*$ with $(y_0,\ldots,y_\ell) \in \breve{f}_*(x)$ iff $\exists i \leq \ell \ y_i \in f(x)$ is computable.
\item We call $f : \mathbf{X} \mto \mathbf{Y}$ eventually-finitely guessable, if the map $\breve{f}_{<\omega} : \mathbf{X} \mto \overline{\mathbf{Y}}^{<\omega}$ with $(y_0,\ldots,y_\ell) \in \breve{f}_{<\omega}(x)$ iff $\exists i \leq \ell \ y_i \in f(x)$ is computable.
\end{enumerate}
\end{definition}

\begin{proposition}
\label{prop:guessabledownwards}
Being $k$-guessable, being finitely guessable and being eventually-finitely guessable are all closed downwards under Weihrauch reducibility.
\begin{proof}
A Weihrauch reduction $f \leqW g$ lifts to Weihrauch reductions $\breve{f}_k \leqW \breve{g}_k$, $\breve{f}_* \leqW \breve{g}_*$ and $\breve{f}_{<\omega} \leqW \breve{g}_{<\omega}$. For this, the inner reduction witness is kept as it is. The outer reduction witness is applied component-wise. For this, it is important that we use the completions of the codomain, as this allows us to extend the outer reduction witness to a total computable map.
\end{proof}
\end{proposition}

Clearly being $k$-guessable for some $k$ implies being finitely guessable, which in turn implies being eventually-finitely guessable. The problem $\C_\mathbb{N}$ is eventually-finitely guessable, but not finitely guessable. The problem $\lpo^*$ is finitely guessable, but for no $k \in \mathbb{N}$ it is $k$-finitely guessable. We will see in Section \ref{sec:guessabilitytt} that these notions can also be separated by $^1(\TT_2)$ and $^1(\TT_3)$. The following will establish a convenient example of a problem which is not $k$-guessable:

\begin{proposition}
\label{prop:accnknotkguessable}
$\accn^k$ is not $k$-guessable.
\begin{proof}
Assume that $\accn^k$ were $k$-guessable, i.e.~that there was a computable function $F$ producing an element of $\left (\overline{(\mathbb{N}^k)}\right )^k$ given an $\accn^k$-instance such that some component $(n_0,\ldots,n_{k+1}) \in \mathbb{N}^k$ constitutes a correct answer to this instance. We can begin building a neutral $\accn$-instance and always extend in a way that excludes no solutions while monitoring how $F$ acts on it. If we find that the $i$-th output of $F$ belongs to $\mathbb{N}^k$, with $n_i$ being the $i$-th number in this tuple, we make $n_i$ the wrong answer to the $i$-th $\accn$-instance we are building. This way, we can diagonalize against any guess that $F$ can make.
\end{proof}
\end{proposition}

\begin{proposition}
The following are equivalent for $f : \mathbf{X} \mto \mathbf{Y}$:
\begin{enumerate}
\item $f$ is $k$-guessable.
\item There exists some $g : \subseteq \Baire \to \mathbf{k}$ with $f \leqW g$.
\end{enumerate}
\begin{proof}
To prove $1 \Rightarrow 2$, fix a computable realizer $F$ of $\breve{f}_k$. Now let $g : \subseteq \Baire \to \mathbf{k}$ return on input $p \in \dom(\delta_\mathbf{X})$ the least $i$ such that the $i$-th component of $F(p)$ is a valid solution to $f(\delta_\mathbf{X}(p))$. It is clear that $f \leqW g$.

To prove $2 \Rightarrow 1$, we observe that $g : \subseteq \Baire \to \mathbf{k}$ is trivially $k$-guessable, which is then inherited by $f$ via Proposition \ref{prop:guessabledownwards}.
\end{proof}
\end{proposition}

\begin{proposition}
The following are equivalent for $f : \mathbf{X} \mto \mathbf{Y}$:
\begin{enumerate}
\item $f$ is finitely guessable.
\item There exists some $g$ with $f \leqW \operatorname{Fin}(g)$
\end{enumerate}
\begin{proof}
To prove the implication $1 \Rightarrow 2$, fix a computable realizer $F$ of $\breve{f}_k$. Now consider the map $g : \subseteq \Baire \to \bigsqcup_{k in \mathbb{N}} \mathbf{k}$ which maps $p \in \dom(F)$ to $i \in \mathbf{k}$ for some $i$ such that $F$ makes $k$ guesses on $p$, and the $i$th guess is correct. It is clear that $f \leqW g$, and that $\operatorname{Fin}(g) \equivW g$. For the implication $2 \Rightarrow 1$, it suffices to observe that $\operatorname{Fin}(g)$ is finitely-guessable almost by construction, and the claim then follows via Proposition \ref{prop:guessabledownwards}.
\end{proof}
\end{proposition}

\paragraph*{Interaction with the algebraic operations}
It is easy to see that if $f$ and $g$ are $k$-guessable, finitely guessable or eventually-finitely guessable then so are $f \sqcup g$ and $f \sqcap g$. We have that $f \times g$ will be (eventually-)finitely guessable if $f$ and $g$ are, but of course being $k$-guessable is not preserved by products. It follows that being (eventually-)finitely guessable is preserved by $^*$. On the other hand $\lpo'$ is $2$-guessable, but $(\lpo')^{\mathrm{u}*} \equivW (\lpo')^\diamond \equivW \Pi^0_2\C_\mathbb{N}$ is not even eventually-finitely guessable as observed by Gill \cite{gill-arxiv}.

\begin{proposition}
\label{prop:guessabilityandcomposition}
\begin{enumerate}
\item If $f$ is $k$-guessable and $g$ is $\ell$-guessable, then $f \star g$ is $k\ell$-guessable.
\item If $f$ is $k$-guessable for some $k \in \mathbb{N}$ and $g$ is finitely guessable, then $f \star g$ is finitely guessable.
\item There is a finitely guessable $f$ and $2$-guessable $g$ such that $f \star g$ is not finitely guessable.
\item If $f$ is finitely guessable and $g$ is eventually-finitely guessable, then $f \star g$ is eventually-finitely guessable.
\item \cite{gill-arxiv} There is an eventually-finitely guessable $f$ and a $2$-guessable $g$ such that $f \star g$ is not eventually-finitely guessable.
\end{enumerate}
\begin{proof}
For all positive claims, we point out that computing the guesses themselves is straight-forward. As the guesses are elements of the completion of the original codomain, there are no obstacles pertaining to definedness. The issue is to obtain the required information about how many guesses are needed. For $1$ this is straight-forward. For $2$, we can run $\breve{g}_*$ on the provided input until we see the number of guesses made, and then multiply that number by $k$ to obtain how many guesses we want to use for $f \star g$. For $4$, we run $\breve{g}_{<\omega}$ which keeps making guesses. For each guess, we start running $\breve{f}_*$ on it and see whether $\breve{f}_*$ ever specifies how many guesses it is going to make. If this never happens, the guess was wrong anyway, and can be ignored. If $\breve{f}_*$ specifies a number of guesses, we use those to compute some guesses for $f \star g$. Overall, this yields a finite number of guesses, which is good enough.

It remains to provide a counter-example for $3$. For that, consider the map $\operatorname{Tmin} : \mathcal{O}(\mathbb{N}) \mto \mathbb{N}$ which returns the minimal element given an enumeration of a non-empty set, and any natural number given an enumeration of the empty set. We find that $\operatorname{Tmin}$ is not finitely-guessable (as specifying the number of guesses before we have seen any number enumerated into the input is not safe, essentially preempting the argument of Proposition \ref{prop:delayableguessable} below). However, we do have $\operatorname{Tmin} \leqW \lpo^* \star \lpo$ -- we use the single $\lpo$ instance to determine whether the set is empty or not, and then if the set is non-empty, use an $\lpo$ instance each to determine whether the finitely many numbers below the first encountered one will appear or not.
\end{proof}
\end{proposition}

\begin{proposition}
\label{prop:eventuallyfinitelyguessablecn}
The following are equivalent for $f : \mathbf{X} \mto \mathbf{Y}$:
\begin{enumerate}
\item $f$ is eventually-finitely guessable.
\item There is some finitely guessable $g$ such that $f \leqW g \star \C_\mathbb{N}$
\end{enumerate}
\begin{proof}
To prove $1 \Rightarrow 2$, let $F$ be a computable realizer of $\breve{f}_{<\omega}$. Let $C : \subseteq \Baire \to \mathbb{N}$ be the map that returns how many guesses $F$ makes on a given name for an input for $f$. We have that $C \leqW \C_\mathbb{N}$. If we let $g : \subseteq \Baire \times \mathbb{N} \mto \mathbf{Y}$ be defined by $g(p,C(p)) = f(\delta_\mathbf{X}(p))$, then $g$ is easily seen to be finitely guessable -- we can use $F$ for guessing, and have the number of guesses made available as part of the input. We then have $f \leqW g \star C \leqW g \star \C_\mathbb{N}$.

The implication $2 \Rightarrow 1$ follows from Proposition \ref{prop:guessabilityandcomposition} (4) together with the observation that $\C_\mathbb{N}$ is eventually finitely guessable.
\end{proof}
\end{proposition}

\paragraph*{Interaction with fractals}

\begin{proposition}
\label{prop:fractalguessable}
For a fractal $f$, the following are equivalent:
\begin{enumerate}
\item There exists some $k \in \mathbb{N}$ such that $f$ is $k$-guessable.
\item $f$ is finitely-guessable.
\end{enumerate}
\begin{proof}
We only need to prove the implication $2 \Rightarrow 1$. Let $F : \subseteq \Baire \mto \Baire$ witness that $f \equivW F$ is a fractal and let $p \in \dom(F)$. Then a realizer of $\breve{F}_*$ will output $k \in \mathbb{N}$ such that $\breve{F}_*(p) \in (\overline{\Baire})^k$ based on same prefix $w$ of $p$. It follows that $F|_{[w]}$ is $k$-guessable. By assumption we have $F|_{[w]} \equivW F \equivW f$, so $f$ is $k$-guessable.
\end{proof}
\end{proposition}

Essentially the same idea as the preceding proposition also yields the following:
\begin{proposition}
\label{prop:delayableguessable}
The following are equivalent for a problem $f$:
\begin{enumerate}
\item There exists some $k \in \mathbb{N}$ such that $?f$ is $k$-guessable.
\item $?f$ is finitely-guessable.
\end{enumerate}
\begin{proof}
We only need to prove the implication $2 \Rightarrow 1$. If $?f$ is finitely guessable, the corresponding computation will have to specify how many guesses it needs to take on the unspecified input to $?f$. However, the unspecified input can then be specified to any input for $f$, which means that the same number of guesses always works.
\end{proof}
\end{proposition}

\begin{proposition}
\label{prop:closedfractalguessable}
For a closed fractal $f$, the following are equivalent:
\begin{enumerate}
\item There exists some $k \in \mathbb{N}$ such that $f$ is $k$-guessable.
\item $f$ is finitely-guessable.
\item $f$ is eventually finitely-guessable.
\end{enumerate}
\begin{proof}
In light of Proposition \ref{prop:fractalguessable}, we only need to prove the implication $(3) \Rightarrow (2)$. By Proposition \ref{prop:eventuallyfinitelyguessablecn}, if $f$ is eventually-finitely guessable, then there is some finitely-guessable $g$ with $f \leqW g \star \C_\mathbb{N}$. As $f$ is a closed fractal, this implies $f \leqW g$ by \cite[Theorem 2.4]{paulyleroux}. Thus, $f$ inherits being finitely guessable from $g$.
\end{proof}
\end{proposition}

\section{Guessability and $\mathrm{TT}$}
\label{sec:guessabilitytt}
In this section we will show that $^1(\mathrm{TT}^1_2) \leqW \mathrm{RT}^1_+$, which gives a positive answer to \cite[Question 6.1]{damir-reed-manlio}. It follows that $^1(\mathrm{TT}^1_2)$ is finitely guessable, but we proceed to prove that it is not $k$-guessable for any $k \in \mathbb{N}$. As a corollary, we can conclude that $\mathrm{TT}^1_2$ is not eventually-finitely guessable (which is already known due to Gill \cite{gill-arxiv}). This furthermore has the corollary that $^1(\mathrm{TT}^1_2) \leW \mathrm{TT}^1_2$ (which was already known due to Dzhafarov, Solomon and Valenti \cite{damir-reed-manlio}), and that $^1(\mathrm{TT}^1_2)$ is not a fractal (and thus that fractality is not preserved by taking first-order parts).  We also show that $^1(\mathrm{TT}^1_3)$ is not finitely guessable, but it is eventually-finitely guessable by \cite[Theorem 5.7]{damir-reed-manlio}. Our results improve upon \cite[Theorems 4.4 \& 4.6]{damir-reed-manlio}, and use similar ideas in the proofs.

\begin{proposition}
$^1(\mathrm{TT}^1_2) \leqW \mathrm{RT}^1_+$.
\begin{proof}
We are given a $2$-colouring of the full binary tree together with a functional that returns some $n \in \mathbb{N}$ upon reading sufficiently long prefixes of monochromatic full subtrees. Upon finding the first subtree such that the functional returns some $n_0$, we can prepare the input to $\mathrm{RT}^1_+$. Let $v_1,v_2,\ldots,v_{k}$ be the leaves of the monochromatic subtree, and let $b \in \{0,1\}$ be its colour. We note that if the trees below the $v_i$ each have a monochromatic subtree coloured $b$, then $n_0$ is a correct output.

We create an instance $p_i$ of $\mathrm{RT}^1_2$ for each $i \leq k$. We write only $0$s to $p_i$ until we find a $(1-b)$-monochromatic subtree below $v_i$ which is sufficiently large to cause the functional to provide some answer $n_i$ on it. Let the leaves of this subtree $u^i_{j}$ for $j \leq \ell_i$. Now we run two processes in parallel, one for writing $0$s to $p_i$, and one for writing $1$s to $p_i$. If no process writes a digit, we just copy the last digit written. The process for $1$s goes through all vertices below $v_i$. Whenever it found a $(1-b)$-coloured vertex below the current one, it writes a $1$ and moves on to the next vertex. Thus, it will write infinitely many $1$s iff there is no vertex below $v_i$ below which all vertices are coloured $b$. The process for writing $0$s goes through $\ell_i$-tuples of vertices, one below each $u^i_j$ for $j \leq \ell_i$. If it finds a $b$-coloured vertex below one of the vertices in the tuple, it writes a $0$ and replaces that vertex with the new one below the appropriate $u^i_j$. Thus, it will write infinitely many $0$s iff it is not the case that below each $u^i_j$ there exists a vertex below which the colouring is constant $b$.

Now let us assume that one of the $k$-many $\mathrm{RT}^1_2$-instances answers $1$, say the $i$-th one. This means that below $v_i$ we can find a $(1-b)$-coloured tree $T_i$ causing the functional to answer $n_i$ (and in particular, we can find $n_i$). Furthermore, either the process writing $1$s has acted infinitely often, or both the $1$-process and the $0$-process acted only finitely many times. In the former case, we know that the colour $1-b$ appears densely below $v_i$, thus $T_i$ is extendible to a full monochromatic subtree, and $n_i$ is a correct answer to $^1(\mathrm{TT}^1_2)$. In the latter case, we know that below every leaf of $T_i$ there is a vertex below which the colouring is constant $1-b$, thus again $T_i$ is extendible to a full monochromatic subtree, and $n_i$ is a correct answer to $^1(\mathrm{TT}^1_2)$.

The remaining case is the one where all $k$ $\mathrm{RT}^1_2$-instances return $0$. We argue that the fact that the $i$-th $\mathrm{RT}^1_2$-instance returned $0$ means that there is a full $b$-monochromatic subtree below $v_i$, so overall, the first prefix we found is extendible, and thus $n_0$ is a correct answer to $^1(\mathrm{TT}^1_2)$. If we never started writing $1$s to $p_i$, then there is no $(1-b)$-monochromatic full subtree below $v_i$ at all, thus there has to be a $b$-coloured one by the truth of $\mathrm{TT}^1_2$. If the process to write $1$s started, but we wrote infinitely many $0$s afterwards, we know that below some $u^i_j$ the colour $b$ appears densely, which means that there will be a $b$-monochromatic full subtree below that $u^i_j$ (and hence below $v_i$). If the process to write $1$s started, and we write only finitely many $1$s afterwards, there is a vertex below $v_i$ below which $b$ is the only colour to appear at all. Again, the claim follows.
\end{proof}
\end{proposition}

\begin{corollary}
$^1(\mathrm{TT}^1_2)$ is finitely-guessable.
\end{corollary}

We want to show that  $^1(\TT_2)$ is not $k$-guessable for any $k \in \mathbb{N}$. We will do this by providing a suitable lower bound for $^1(\TT_2)$, namely a kind of cofinite choice principle:
\begin{definition}
Let $\ecfc$ be the problem whose input are pairs $k \in \mathbb{N}$, $A \in \mathcal{A}(\mathbb{N})$ with $|\mathbb{N} \setminus A| \leq k$, and whose solutions are any $n \in A$.
\end{definition}

We can easily see that $\accn^* \leqW \ecfc$, as given $k$ instances of $\accn$ there are at most $k$ numbers that do not answer all of them correctly at once. Thus, we obtain a single $\ecfc$ instance whose answer we can return to all $k$ $\accn$ instances.

\begin{proposition}
\label{prop:ecfcleqwtt2}
$\ecfc \leqW \TT_2$.
\begin{proof}
We are given $k \in \mathbb{N}$ and some $A \in \mathcal{A}(\mathbb{N})$ with $|\mathbb{N} \setminus A| \leq k$. For the outer reduction witness, we pick an injective enumeration $(s_n)_{n \in \mathbb{N}}$ of the $k+1$-element antichains in $2^{<\omega}$. If we find that $s_n$ appears in a monochromatic subtree, we return $n$ as solution for $\ecfc$. This means that we may have to prevent up to $k$ antichains of size $k+1$ each to be part of a monochromatic copy of $2^{<\omega}$.

Our colouring $c : 2^{<\omega} \to \{0,1\}$ initially is constant $0$. If $n_0$ gets enumerated outside of $A$, we pick an element of $s_{n_0}$ and make every future descendent of it coloured $1$, elsewhere the colouring remains constant $0$. If another number $n_1$ gets enumerated outside of $A$, we reevaluate the colouring choices. We pick an element $t_0$ of $s_{n_0}$ and an element $t_1$ of $s_{n_1}$ such that $t_0$ is incomparable with some element of $s_{n_1}$ other than $t_1$, and $t_1$ is incomparable with an element of $s_{n_0}$ other than $t_0$. For future colours, every descendent of $t_0$ or $t_1$ gets coloured $1$, every other vertex gets coloured $0$. We proceed in this way for any future element that gets removed from $A$. As this can happen at most $k$ times, we can always find elements of the antichains incomparable with an unchosen vertex in each other antichain.

After every element that is going to be removed from $A$ has been removed, the colouring is locally constant. Moreover, any antichain corresponding to a wrong answer has at least one vertex below which the colouring is eventually constant $1$ and another below which the colouring is eventually constant $1$. This prevents any such antichain from appearing in a monochromatic copy of $2^{<\omega}$, i.e.~the reduction works correctly.
\end{proof}
\end{proposition}

\begin{corollary}
\label{corr:tt12firstordernotkguessable}
For no $k \in \mathbb{N}$ does it hold that $^1(\TT_2)$ is $k$-guessable.
\begin{proof}
Proposition \ref{prop:ecfcleqwtt2} in particular shows that $\accn^k \leqW\ ^1(\TT_2)$ for all $k \in \mathbb{N}$. By Proposition \ref{prop:accnknotkguessable}, this shows that  $^1(\TT_2)$ is not $k$-guessable.
\end{proof}
\end{corollary}

\begin{corollary}
\label{corr:rt1knleqttk}
$\TT_2 \nleqW \RT^1_k$ for all $k \in \mathbb{N}$.
\end{corollary}

\begin{corollary}
$\mathrm{TT}^1_2$ is not eventually-finitely guessable.
\begin{proof}
As $\mathrm{TT}^1_2$ is a closed fractal, if it were eventually-finitely guessable, it would already be $k$-guessable for some $k \in \mathbb{N}$ by Proposition \ref{prop:closedfractalguessable}. This would then be inherited by $^1(\mathrm{TT}^1_2)$ by Proposition \ref{prop:guessabledownwards}, contradicting Corollary \ref{corr:tt12firstordernotkguessable}.
\end{proof}
\end{corollary}

\begin{corollary}
$^1(\mathrm{TT}^1_2)$ is not $\sigma$-join-irreducible (and in particular, not a fractal).
\end{corollary}

The following lemma makes precise an argument by Dzhafarov, Solomon and Valenti \cite{damir-reed-manlio}:
\begin{lemma}
\label{lemma:tt12tt13delayable}
$?(^1(\mathrm{TT}^1_2)) \leqW\ ^1(\mathrm{TT}^1_3)$
\begin{proof}
As long as the $?(^1(\mathrm{TT}^1_2))$ is unspecified, we let the input to $^1(\mathrm{TT}^1_3)$ be a tree which is $2$-monochromatic, together with the functional that returns $0$ upon seeing a single $2$-coloured vertex. If the input to $?(^1(\mathrm{TT}^1_2))$ ever gets specified, we make the remainder of the colouring fed to $^1(\mathrm{TT}^1_3)$ agree with the colouring for $^1(\mathrm{TT}^1_2)$ -- which means that there cannot be a $2$-coloured full subtree anymore -- and let the functional return $n+1$ on any $0$ or $1$ monochromatic subtree whenever the functional received as input would return $n$. We then know that if we receive $0$ as answer from $^1(\mathrm{TT}^1_3)$, the input to $?(^1(\mathrm{TT}^1_2))$ was never specified, and if we receive $n+1$ from $^1(\mathrm{TT}^1_3)$, we should answer $n$ to $?(^1(\mathrm{TT}^1_2))$.
\end{proof}
\end{lemma}

\begin{corollary}
$^1(\mathrm{TT}^1_3)$ is not finitely-guessable.
\begin{proof}
By Lemma \ref{lemma:tt12tt13delayable}, if $^1(\mathrm{TT}^1_3)$ were finitely-guessable, then the same would hold for $?(^1(\mathrm{TT}^1_2))$. By Proposition \ref{prop:delayableguessable}, then there has to be some $k$ such that $?(^1(\mathrm{TT}^1_2))$ and thus $^1(\mathrm{TT}^1_2)$ is $k$-guessable. This contradicts Corollary \ref{corr:tt12firstordernotkguessable}.
\end{proof}
\end{corollary}

\section{The finitary part of $\TT$}
The $k$-finitary part of a principle, denoted by $\operatorname{Fin}_k(f)$, is the greatest Weihrauch degree below $f$ having a representative with codomain $\mathbf{k}$ \cite{paulycipriani1}. We abbreviate $\operatorname{Fin}(f) = \bigsqcup_k \operatorname{Fin}_k(f)$. 
This means that $\Fin(f)$ characterizes which problems with finite codomain can be solved by $f$. In \cite{paulycipriani1}, this notion was introduced in order to prove separations between Weihrauch degrees by reasoning in a simpler, more restrictive setting. However, the finitary part of a problem can also be used to show similarities between problems. For example, in  \cite{gohpaulyvalenti2-cie} it is shown that the problem of finding a descending sequence in an ill-founded linear order and the problem of finding a bad sequence in a non-wellorder have the same finitary part, although their first-order parts differ. Gill already proposed to study the finitary part of $\TT_k$, which we characterize in this section. On the one hand, our result lets us answer a question by Dzhafarov, Solomon and Valenti in the negative by proving $\C_3 \nleqW \TT_2$; on the other hand we will see that $\TT_k$ and $\RT^1_k$ have the same finitary part; and thus the additional strength of $\TT_k$ over $\RT^1_k$ only materializes when we consider problems with a more complicated codomain.

\begin{proposition}
\label{prop:fintcn}
$\operatorname{Fin}_k(\mathrm{TC}^*_\mathbb{N}) \equivW \mathrm{RT}^1_k$
\begin{proof}
Since $\mathrm{RT}^1_{k+1} \leqW (\mathrm{RT}^1_2)^k \leqW \mathrm{TC}^k_\mathbb{N}$, the right-to-left reduction is clear. For the left-to-right reduction, assume that $ f: \subseteq \Baire \mto \mathbf{k}$ satisfies $f \leqW \mathrm{TC}_{\mathbb{N}}^*$. We want to prove that already $f \leqW \mathrm{RT}^1_k$.

We construct the input to $\mathrm{RT}^1_k$ by observing how the reduction $f \leqW \mathrm{TC}_{\mathbb{N}}^*$ works on some input $x \in \dom(f)$. We can first determine the number $j$ of copies of $\mathrm{TC}_\mathbb{N}$ which is used. Starting with $v = (0,0,\ldots,0) \in \mathbb{N}^j$, we simultaneously search for whether the outer reduction witness produces some $i \in \mathbf{k}$ on input $(x,v)$, in which case we write $i$ to the $\mathrm{RT}^1_k$-instance, and whether $v_j$ gets removed from the $j$-th $\mathrm{TC}_\mathbb{N}$-instance produced by the inner reduction witness, in which case we increment $v_j$ by $1$. While waiting for additional numbers to print to the $\mathrm{RT}^1_k$-instance, we repeat the latest one.

It remains to argue that any number $i < k$ which we print infinitely often is a correct solution to $f(x)$. We observe that for any $\mathrm{TC}_\mathbb{N}$-input produced from $x$ which is not enumerating all numbers we will eventually reach a value $v_j$ which remains a correct output. All other $\mathrm{TC}_\mathbb{N}$-instances can return any number anyway. Thus, once we have reached a correct $v_j$ value for all specified positions, all future values $v$ can take are indeed correct solutions for the $\mathrm{TC}_\mathbb{N}^j$-instance queried for $x$, hence the value $i < k$ returned by the outer reduction witness are correct solutions to $f(x)$.
\end{proof}
\end{proposition}

\begin{corollary}
$\operatorname{Fin}(\mathrm{TT}^1_+) \equivW \mathrm{RT}^1_+$
\begin{proof}
By combining Proposition \ref{prop:fintcn} and Proposition \ref{prop:tttcn}.
\end{proof}
\end{corollary}

\begin{corollary}
\label{corr:isthereamonochromatictree}
The problem ``Given a colouring $c : 2^{<\omega} \to \mathbf{k}$, find some $i \in \mathbf{k}$ such that there exists a copy of $2^{<\omega}$ coloured $i$'' is equivalent to $\RT^1_k$.
\end{corollary}

The following is the main result of this section, it fully characterizes the $j$-finitary part of any $\TT_k$. 
\begin{theorem}
\label{theo:finitarypart}
$\Fin_j(\TT_k) = \RT^1_{\min\{j,k\}}$ for all $j,k \in \mathbb{N}$.
\end{theorem}

Before we prove the theorem, we state the following consequences:

\begin{corollary}
\label{corr:finttkisrtk}
$\Fin(\TT_k) \equivW \RT^1_k$
\end{corollary}

The preceding corollary together with Corollary \ref{corr:isthereamonochromatictree} states that for computing multi-valued functions with finite codomain we cannot do anything stronger with $\TT_k$ than coding the into for what colours there are monochromatic copies of $2^{<\omega}$. By Proposition \ref{prop:ecfcleqwtt2}, this does not extend to multi-valued functions with discrete codomain.

\begin{corollary}
$\C_3 \nleqW \TT_2$
\begin{proof}
By Corollary \ref{corr:finttkisrtk}, as $\C_3$ has finite codomain, if $\C_3 \leqW \TT_2$ then $\C_3 \leqW \RT^1_2$. The latter is easily seen to be false.
\end{proof}
\end{corollary}

The following recovers a result already established by Gill:
\begin{corollary}
\label{corr:cnnleqwttk}
$\C_\mathbb{N} \nleqW \TT_k$ for all $k \in \mathbb{N}$.
\begin{proof}
For all $k \in \mathbb{N}$ it holds that $\C_{k+1} \leqW \C_\mathbb{N}$, but $\C_{k+1} \leqW \TT_k$ would imply $\C_{k+1} \leqW \RT^1_k$ by Corollary \ref{corr:finttkisrtk}, which is easily seen to be false.
\end{proof}
\end{corollary}

To prove Theorem \ref{theo:finitarypart}, we use the notion of a \emph{commit tree} which describes how a putative reduction is acting on a prefix:

\begin{definition}
We are given a colouring $c : 2^{<\omega} \to \mathbf{k}$ and a continuous partial function $H$ from $c$-monochromatic finite subtrees of $2^{<\omega}$ to $\mathbf{j}$.
\begin{enumerate}
\item A \emph{unitary commit tree} for $i \in \mathbf{k}$, $x \in \mathbf{j}$ is a finite tree $T$ of uniform height where each vertex $v \in T$ is labeled by a set of vertices $S_v \subseteq 2^{<\omega}$ all having colour $i$ under $c$ isomorphic to some $2^{\leq \ell_v}$ such that $H$ returns $x$ on reading a prefix of $S_v$. The degree of a non-leaf $v \in T$ is equal to the number of leaves in $S_v$. If $u \in T$ is the $i$-th child of $v \in T$, then all of $S_u$ is below the $i$-th leaf in $S_v$.
\item A \emph{$(n_0,n_1,\ldots,n_m)$-layered commit tree} is a finite tree of uniform height $\sum_{i \leq m} n_i$ where each vertex $v$ is labeled by a $c$-monochromatic set of vertices $S_v \subseteq 2^{<\omega}$ such that the colour of $S_v$ only depends on $L(v) := \min \{i \leq m \mid \sum_{i' \leq i} n_{i'} \geq \operatorname{height}(v)\}$, and moreover, if $L(v) \neq L(u)$, then $S_v$ and $S_u$ have different colour. If $v$ has height $\sum_{i' \leq i} n_{i'}$ for some $i < m$, then the degree of $v$ is $1$, and if $u$ is the unique successor of $v$, then every vertex in $S_u$ is below the root of $S_v$. Otherwise the degree of a non-leaf $v$ is equal to the number of leaves of $S_v$. Finally, for every vertex $v$ we have that $H$ will output some $x \in \mathbf{j}$ upon reading a prefix of $S_v$ where $x$ only depends on $L(v)$.
\item A commit tree $T_1$ is a \emph{condensation} of a commit tree $T_0$ if it is obtained by repeatedly taking a vertex $v \in T_0$ and a successor $v'$ of $v$, and then replacing the subtree rooted at $v$ with the subtree rooted at $v'$.
\end{enumerate}
\end{definition}

A unitary commit tree describes ``plenty'' of opportunities for a Weihrauch reduction to $\tt_k$ from a problem with codomain $\mathbf{j}$ to give a particular answer. A layered commit tree is comprised of several layers of uniform commit trees, and if we have enough layers of sufficient height, one of these opportunities for the reduction actually has to be valid. This idea will be formalized and proven in the following.

\begin{lemma}
Let $T_0$ be a unitary commit tree of height $2^{\ell + 1}$, and let $d$ be a $2$-colouring of the leaves of $T_0$. Then there exists a condensation $T_1$ of $T_0$ having height $2^\ell$ such that $d$ restricted to the leaves of $T_1$ is monochromatic.
\begin{proof}
We construct $T_1$ layer by layer, starting with the leaves. In each stage, we chose which vertices from the next two layers of $T_0$ to keep. (If a vertex is kept, but no suitable ancestors are, it can be lost at a later stage.)

In the first stage, we consider the two layers of vertices in $T_0$ closest to the leaves. Let $v_0,v_1,\ldots,v_m$ be the vertices on the penultimate layer. For each $v_i$, either it has a child all of whose leaves have the same colour; or all of its children have leaves with both colours. In the former case, we replace $v_i$ with such a child. In the latter case, we replace each child of $v_i$ with a leaf below it coloured $0$. The result is a condensation of $T_0$ with its height reduced by $1$ such that every vertex in the bottom layer has only leaves of a single colour.

We then consider the next two layers, using the vertices previously chosen as leaves, with the colour being the one all their leaves share. Repeating this process for $2^\ell$ times total yields the desired result.
\end{proof}
\end{lemma}

\begin{corollary}
\label{corr:condensation}
Let $T_0$ be a unitary commit tree of height $2^{\ell + s}$, and let $d$ be a $2^{s}$-colouring of the leaves of $T_0$. Then there exists a condensation $T_1$ of $T_0$ having height $2^\ell$ such that $d$ restricted to the leaves of $T_1$ is monochromatic.
\end{corollary}

\begin{corollary}
\label{corr:condensationlayered}
Let $T_0$ be a $(2^{\ell_0 + s},2^{\ell_1 + s},\ldots,2^{\ell_m + s})$-layered commit tree, and let $d$ be a $2^s$-colouring of the leaves of $T_1$. Then there exists a  $(2^{\ell_0},2^{\ell_1},\ldots,2^{\ell_m})$-layered commit tree $T_1$ such that $T_1$ is a condensation of $T_0$, and $d$ restricted to the leaves of $T_1$ is monochromatic.
\end{corollary}

\begin{lemma}
\label{lemma:committreefinish}
Given a colouring $c : 2^{<\omega} \to \mathbf{k}$ and a $(2^{n_0},\ldots,2^{n_{k-1}})$-layered commit tree $T$ such that $n_i > \lceil \log k \rceil$, there exists some $v \in T$ such that $S_v$ is extendible to a $c$-monochromatic copy of $2^{<\omega}$; and moreover we can compute such a $v$ given a labelling $d$ of the leaves of $T$ such that there exists a $c$-monochromatic copy of $2^{<\omega}$ of colour $d(u)$ below the root of $S_u$.
\begin{proof}
By Corollary \ref{corr:condensationlayered}, there is a condensation $T_1$ of $T$ such that $d$ is constant on the leaves of $T_1$, and $T_1$ is a $(n_0',n_1',\ldots,n_{k-1}')$-layered commit tree with $n'_i \geq 2$. Let $\ell \in \mathbf{k}$ be constant value returned by $d$. Then any non-leaf $u \in T_1$ such that $S_u$ has colour $\ell$ has the property that below any leaf of $S_u$ there exists a monochromatic copy of $2^{<\omega}$ of colour $\ell$, which means that we can extend $S_u$ to a monochromatic copy of $2^{<\omega}$.
\end{proof}
\end{lemma}

\begin{proof}[Proof of Theorem \ref{theo:finitarypart}]
In the case that $j \leq k$, the claim already follows from Propositions \ref{prop:fintcn}, \ref{prop:tttcn} together with the fact that $\RT^1_j \leqW \TT_k$ for $j \leq k$. Thus, we can focus our attention to the case that $j > k$.

Assume that $f : \subseteq \Baire \mto \mathbf{j}$ satisfies $f \leqW \TT_k$, with $H$ as the outer reduction witness and $K$ as the inner. We are given an input $p$ for $f$, and consider the colouring $c = K(p)$ and view $H$ as a continuous partial function producing elements of $\mathbf{j}$ given sufficiently long prefixes of monochromatic subtrees of $2^{<\omega}$. We will enumerate a sequence of up to $k$ commit trees such that the $i$-th tree we produce has $i$ layers, and such that the functions mapping $L(v)$, the layer of a vertex $v$, to which $\ell \in \mathbf{j}$ we get from applying $H$ to $S_v$, agree on their common domains for all the trees we built. We will then show that using $\mathrm{KL}$, we can compute some $i$ such that one of the commit trees we built has a vertex $v$ with $L(v) = i$ such that $S_v$ is extendible to a $c$-monochromatic copy of $2^{<\omega}$. Since this is a problem with codomain $\mathbf{k}$, and since $\Fin_k(\mathrm{KL}) \equivW \RT^1_k$ \cite[Theorem 13]{paulypradicsolda}, actually $\RT^1_k$ suffices to identify $i$. Moreover, given $i$ we can compute some $\ell \in \mathbf{j}$ such that there exists a $c$-monochromatic copy of $2^{<\omega}$ on a prefix of which $H$ would return $\ell$. This establishes $\Fin_j(\TT_k) \leq \RT^1_k$, as we set out to do.

By appealing to the tree theorem itself, we see that there are unitary commit trees of arbitrary height. Let $h > \lceil \log (kj) \rceil$. We can thus search for some $i_0 \in \mathbf{k}$ and $\ell_0 \in \mathbf{j}$ and a unitary commit tree $T_0$ of height $2^{1+kh}$. We can view $T_0$ equivalently as a $(2^{1+kh})$-layered commit tree.

We then search below each root of some $S_v$ for a leaf $v$ of $T_0$ for a unitary commit tree of height $2^{1+h(k-1)}$ for a colour other than $i_0$. If this search never succeeds, $T_0$ is the only layered commit tree we construct. Otherwise, we can label the leaves of $T_0$ by the pair $(i,\ell)$ such that we found a commit tree for $i,\ell$ below it. By Corollary \ref{corr:condensation} we can condense $T_0$ to some $T_0'$ of height $2^{1+h(k-1)}$ such that the same choice $i_1,\ell_1$ works for all remaining leaves. We then built a $(2^{1+h(k-1)},2^{1+h(k-1)})$-layered commit tree $T_1$ by appending the unitary commit trees we found to the leaves of $T_0'$.

The process continues, with us now searching for unitary commit trees of height $2^{1+h(k-2)}$ below each leaf $v$ of $T_1$ with a colour other than $i_0$ or $i_1$, and so on. As there are only $k$ colours, we built at most $k$ layered commit trees. Moreover, being on the $b$-th layer of any of the layered commit trees we built means that the answer is $\ell_b \in \mathbf{j}$, as intended.

As $\mathrm{KL} \equivW \widehat{\RT^1_k}$, by Corollary \ref{corr:isthereamonochromatictree} we can use $\mathrm{KL}$ to obtain from $c : 2^{<\omega} \to \mathbf{k}$ some $d : 2^{<\omega} \to \mathbf{k}$ such that $d(v) = i$ means that there exists a $c$-monochromatic copy of $2^{<\omega}$ below $v$ with colour $i$. As $T_0$ is always defined, we can compute it and inspect the colour $d$ assigns to the leaves of the $S_u$ where $u$ ranges over the leaves of $T_0$. If there is a leaf $u$ such that all leaves of $S_u$ get assigned $i_0$ by $d$, then $S_u$ extends to a monochromatic copy of $2^{<\omega}$, and thus $0$ is a valid answer to ``Which layer contains an extendible finite subtree'', which is what we are trying to solve. If there is no such leaf, the search leading to the construction of $T_1$ will be successful. We then inspect $d$ on the leaves of the $S_u$ where $u$ ranges over the leaves of $T_1$; and will be able to either identify some $S_u$ which is extendible (either on the first or second layer), or we are assured that $T_2$ is going to be well-defined. By Lemma \ref{lemma:committreefinish}, we will find something extendible at the very least once we have constructed $T_{k-1}$.
\end{proof} 

\section*{Acknowledgements}
I am grateful to Damir Dzhafarov, Kenneth Gill and Manlio Valenti for discussions leading to the results in this note.

\bibliographystyle{eptcs}
\bibliography{spieltheorie}

\end{document}